\documentclass[11pt]{amsart}

\usepackage{a4wide, amsmath, amsfonts, amssymb, mathrsfs, amsthm, breqn,color}
\usepackage[hyperfootnotes=false]{hyperref}

\allowdisplaybreaks[2]

\numberwithin{equation}{section}

\newtheorem{theorem}{Theorem}[section]
\newtheorem{lemma}[theorem]{Lemma}

\newtheorem{remark}[theorem]{Remark}
\newtheorem{proposition}[theorem]{Proposition}

\newcommand{\dd}{\,\mathrm{d}}

\newcommand{\R}{\mathbb{R}}
\newcommand{\N}{\mathbb{N}}

\renewcommand{\epsilon}{\varepsilon}

\title[SDEs with fBm: dependence on the Hurst parameter]{Stochastic differential equations driven by fractional Brownian motion: dependence on the Hurst parameter}

\author[Kwossek]{Anna P. Kwossek}
\address{Anna P. Kwossek, University of Vienna, Austria}
\email{anna.paula.kwossek@univie.ac.at}

\author[Neuenkirch]{Andreas Neuenkirch}
\address{Andreas Neuenkirch, University of Mannheim, Germany}
\email{neuenkirch@uni-mannheim.de}

\author[Pr{\"o}mel]{David J. Pr{\"o}mel}
\address{David J. Pr{\"o}mel, University of Mannheim, Germany}
\email{proemel@uni-mannheim.de}

\date{\today}

\begin{document}
	
	\begin{abstract}
				Stochastic models with fractional Brownian motion  as source of randomness have become popular since the early 2000s. 
		 Fractional Brownian motion (fBm) is a Gaussian process, whose covariance depends on the so-called Hurst parameter $H\in (0,1)$. Consequently, stochastic models with fBm also depend on the Hurst parameter $H$, and the stability of these models with respect to $H$ is an interesting and important question. In recent years, the continuous (or even smoother) dependence on the Hurst parameter has been studied for several stochastic models, including stochastic integrals with respect to fBm, stochastic differential equations (SDEs) driven by fBm  and  also stochastic partial differential equations with fractional noise, for different topologies, e.g., in law or almost surely, and for finite and infinite time horizons.
			In this manuscript, we give an overview of these results with a particular focus on SDE models.
	\end{abstract}
	
	\maketitle
	
	

\date{\today}

\maketitle

\section{Introduction}

Fractional Brownian motion (fBm) is  a  Gaussian process $B^H=(B^H_t)_{t \geq 0}$  with continuous sample paths, zero mean, i.e.,
 $$  \mathbf{E} \big [B^H_t \big]=0, \qquad t \geq 0,$$
 and covariance
$$  \mathbf{E} \big[ B^H_sB^H_t \big] =\frac{1}{2}\big(|s|^{2H}+|t|^{2H}-|t-s|^{2H}\big), \qquad s,t \geq 0.$$
The parameter $H \in (0,1)$ is the so-called Hurst parameter.
Fractional Brownian motion was introduced by Kolmogorov  \cite{kolmogorov} in 1940, and it was reinvented
and popularized  by  Mandelbrot and van Ness  \cite{MandelbrotVanNess}  in 1968. 
In \cite{MandelbrotVanNess} also the terms `fractional Brownian motion' and `Hurst parameter'  were coined, and moreover a representation of fBm as a Volterra process with a  kernel function depending on $H$  was given.

Note that  for $H=\frac{1}{2}$  fractional Brownian motion coincides with the standard Brownian motion, since in this case we have
$$  \mathbf{E} \big{[} B^{\frac{1}{2}}_s B^{\frac{1}{2}}_t \big ] = \min \{s,t\}, \qquad s,t \geq 0.$$
While fBm possesses some properties similar to Brownian motion, as, e.g., self-similarity and stationary increments, fBm is neither a semi-martingale nor a Markov process for $H \neq \frac{1}{2}$.  Moreover, the increments of fBm are negatively correlated if $H<\frac{1}{2}$ and positively correlated   for $H>\frac{1}{2}$. In the latter case, the increments of fBm also exhibit a long-range dependence property.  Due to the classical Kolmogorov continuity theorem almost all sample paths of fBm are H\"older continuous of all orders $\lambda <H$ on compact time intervals, i.e., for all $T>0 $ and almost all $\omega \in \Omega$ we have
$$  \sup_{s,t \in [0,T] } \frac{|B^H_t(\omega)-B^H_s(\omega)|}{|t-s|^{\lambda}} < \infty.$$
For these and further properties, see, e.g., Chapter 5 in \cite{nualartbook} or Chapter 2 in \cite{nourdin_book}.

\smallskip

Fractional Brownian motion has found applications in various fields, which include hydrology (cf.~ \cite{Hurst,hydro}),
turbulence models (cf.~\cite{Fann_Komm,Flandoli}),  electrical engineering (cf.~\cite{DMS,DR}), biophysics (cf.~\cite{ko,ks,Richard:train})
and mathematical finance (cf.~\cite{guasoni,Bender,gatheral_bayer,gatheral1}).
Consequently, various stochastic models with fBm as noise input have been studied in recent years, e.g., stochastic integrals, stochastic differential equations (SDEs) and  also stochastic partial differential equations.
Since fractional Brownian motion depends on the Hurst parameter $H$, it is a natural question to analyse the dependence of these models on the Hurst parameter. In particular, stability results which quantify the deviation of the fractional Brownian model (with $H \neq \frac{1}{2}$) from the standard Brownian model (with $H=\frac{1}{2}$) are of obvious interest. 
Moreover, such a stability analysis is also useful for statistical applications, see, e.g.,  \cite{HRstat}.

Before we turn to SDEs driven by fBm, note that the question of the dependence on the Hurst parameter is well-posed. For fBm itself, we have trivially a smooth dependence of its law:
For $t>0$, $f \colon \mathbb{R} \rightarrow \mathbb{R}$ measurable and of exponential growth the map
$$ (0,1) \ni H \, \mapsto \, \mathbf{E} [f(B^H_t)] \in \mathbb{R}$$
is infinitely differentiable, since we have the representation
$$ \mathbf{E} [f(B^H_t)] = \frac{1}{\sqrt{2\pi t^{2H}}}\int_{\mathbb{R}}f(x) \exp \left( -\frac{x^2}{2t^{2H}}\right) \dd x.$$ Moreover, in the context of constructing  multi-fractional Brownian motions it has been shown (Theorem 4 in \cite{pelletier}) that the Mandelbrot--van Ness representation of fBm, see \cite{MandelbrotVanNess} and  the following Section \ref{sec:MvN}, satisfies
$$ \lim_{\delta \rightarrow 0} \, \sup_{ \substack{H,H' \in [a,b] \\ |H-H'| \leq \delta} } \,  \sup_{t \in [0,T]} | B_t^H-B_t^{H'}| =0  $$ almost surely
for any $0<a<b<1$ and $T>0$. Thus we have (at least) a uniform continuous dependence on the Hurst parameter for almost all sample paths on $[0,T]$.

\medskip

\subsection{SDEs driven by fBm}

Stochastic differential equations driven by fractional Brownian motion are usually treated as pathwise integral equations. For the driving  $m$-dimensional fractional Brownian  motion $B^H$, i.e., $$ B^H= \begin{pmatrix}B^{H,1} \\   {}^{\vdots} \\  B^{H,m}  \end{pmatrix},$$ where $B^{H,1}, \ldots, B^{H,m}$ are independent copies of a one-dimensional fractional Brownian motion with Hurst parameter $H$, one 
considers the sample paths $g_t=B^H_t(\omega)$, $ t \geq 0,$ for fixed $\omega \in \Omega$, and analyses the deterministic integral equation
\begin{align}\label{det_int}
	x_t = x_0+  \int_0^t \mu(x_s) \dd s +  \int_0^t \sigma(x_s) \dd g_s, \quad t \geq 0,
\end{align}
with $x_0 \in \R^d$ and sufficiently smooth $\mu \colon \R^d \rightarrow \R^d$, $\sigma \colon \R^d \rightarrow \R^{d \times m}$. 
The latter equation reads in componentwise form as
\begin{align*}
	x_t^i = x_0^i+  \int_0^t \mu^i(x_s) \dd s +  \sum_{j=1}^m \int_0^t \sigma^{ij}(x_s) \dd g^j_s, \quad t \geq 0, \quad i=1, \ldots, d.
\end{align*}
Based on the dimension $m$ of the driving fBm and the properties of the coefficients $\mu$ and $\sigma$, different  techniques have been applied to analyse SDEs driven by fBm, see, e.g., \cite{Doss,Sussmann,NR,CoutinQian,gubinelli,fractional_rough, fractional_regular,FrizVictoir10} and the monographs \cite{Lyons_Qian,FrizVictoir,FrizHairer}. In particular, for $m>1$ and non-additive noise the choice of the technique is dictated by the path regularity of the driving fBm
measured in terms of its $p$-variation- or H\"older-regularity. For $H>1/2$ the classical Young integration theory \cite{Young} can be used here, while for $1/4<H\leq 1/2$  one typically relies on rough path theory, which was initiated by T. Lyons in \cite{TL1,TL2}. Several variants of the rough path approach have been now developed, see, e.g., \cite{Davie}, \cite{fractional_rough}, especially the controlled rough path approach of M. Gubinelli \cite{gubinelli}.

As a consequence of the pathwise approach, the dependence of the solution of
\begin{align*}%\label{SDE_int}
	X_t^H= x_0+  \int_0^t \mu(X^H_s) \dd s +  \int_0^t \sigma(X^H_s) \dd B^H_s, \quad  t \geq 0,
\end{align*}
on the Hurst parameter $H$ can be analysed in  many cases in terms of the smoothness of the solution map $x=\Gamma(g)$ of the equation \eqref{det_int} with respect to the driving path $g$.

Of particular interest are the following questions:

\begin{itemize}
	\item[(1)] Does for a given $t>0$ and suitable $\mathcal{H}\subset(0,1)$ the law of $X_t^H$ depend continuously on $H$? And if yes, does the map 
	$$ \mathcal{E}_{\varphi,\mathcal{H},t}\colon \mathcal{H} \rightarrow \R, \qquad \mathcal{E}_{\varphi,\mathcal{H},t}(H)= \mathbf{E} \big[ \varphi(X_t^H) \big],$$ exhibit Lipschitz properties  (for sufficiently regular $\varphi \colon \R^d  \rightarrow \R$)? 
	\item[(2)] Is  $X_t^H(\omega)$  for fixed $\omega \in \Omega$ differentiable or (globally) Lipschitz continuous with respect to $H$? Or put more precisely: does for
	fixed $\omega \in \Omega$,  suitable $\mathcal{H}\subset(0,1)$ and $t>0$ the map
	$$ \mathcal{S}_{\mathcal{H},t}\colon  \Omega \times  \mathcal{H} \rightarrow \R, \qquad \mathcal{S}_{\mathcal{H},t}(\omega,H)=X^H_t(\omega),$$
	exhibit these smoothness properties?
\end{itemize}
We will discuss these and related questions for different types of SDEs in the remainder of the manuscript. Note that for the second question we will always use and rely on the  Mandelbrot--van Ness representation from Section \ref{sec:MvN}. Our aim is to give an overview of the main results; the required mathematical techniques  are manifold and for them we will provide the appropriate pointers to the literature.

\medskip

\subsection{Setup and standard definitions} In the following we will work always on a complete probability space $(\Omega, \mathcal{A}, \mathbf{P})$ which is rich enough for all objects to be well-defined. We will use the notation $X \stackrel{a.s.}{=} Y$ for the almost sure equality of two random variables $X$ and $Y$ as well as $X \stackrel{\mathcal{L}}{=} Y$ for equality of their laws. Similarly, we will denote by $X_n \stackrel{\mathcal{L}}{\rightarrow} Y$ the convergence in law of the sequence of random variables $X_n$, $n \in \N,$ to $Y$.

For a set $D$, we will denote by $\mathbf{1}_D$  the indicator function of this set. For a vector $x \in \R^d$ we will denote its Euclidean norm by $\| x \|$, while for a matrix $A \in \R^{d \times d}$ the quantity $\|A\|$ is its corresponding operator norm. In the  case $d=1$ we simply write $| \cdot |$ instead of $\| \cdot \|$.
For $\lambda \in [0,1)$, a domain $O \subset \R^d$ and a function  $f\colon \overline{O} \rightarrow \R^{\ell}$ we denote the $\lambda$-H\"older norm by  
 $\| \cdot \|_{\lambda, O}$, i.e., we set
\begin{align*}
	\|f \|_{\lambda,O} \, = \sup_{t \in {O} }\,  \|f(t)\| \, + \sup_{s,t \in {O}} \frac{\|f(t)-f(s)\|}{\|t-s\|^{\lambda}}. 
\end{align*} Moreover, we denote by $C^{\lambda}(\overline{O};\R^{\ell})$ the space of all functions  $f \colon \overline{O} \rightarrow \R^{\ell}$ with a finite $\lambda$-H\"older norm. As usual we define by $C^{k}(\R^{\ell_1};\R^{\ell_2})$ the space of all functions $f \colon \R^{\ell_1} \rightarrow \R^{\ell_2}$ which are $k$-times continuously differentiable. By  $C^{k}_b(\R^{\ell_1};\R^{\ell_2})$ we denote the subset of functions $f \in  C^{k}(\R^{\ell_1};\R^{\ell_2})$ which are bounded together with their derivatives, while $C^{k}_{pol}(\R^{\ell_1};\R^{\ell_2})$ denotes the subset of functions   $ f \in C^{k}(\R^{\ell_1};\R^{\ell_2})$ which are of polynomial growth together with their derivatives.

For completeness, we recall the notion of Fr\'echet differentiability: let 
$ (U,\|{\cdot}\|_U)$ and $(V,\|{\cdot}\|_V)$ be two normed vector spaces and  let $U' \subset U$ be an open subset. An operator $ A \colon U' \to V$ is called Fr\'echet differentiable at $u \in U'$ if there exists a bounded linear operator $A'(u)\colon U  \to V$ such that
$$ \lim_{\|h\|_U\to 0} \frac{1}{\|h\|_U}\, \big\| A(u+h)-A(u)-A'(u)h \big\|_V=0.$$
If $A$ is Fr\'echet differentiable for  all $u \in U'$, then the operator is called Fr\'echet differentiable with Fr\'echet derivative $A'\colon U' \to \mathcal{L}(U,V)$, where $\mathcal{L}(U,V)$ denotes the space of  bounded linear
operators from $U$ to $V$.

\bigskip
\section{The Mandelbrot--van Ness representation of fBm and its dependence on the Hurst parameter}\label{sec:MvN}
 Let  $B=(B_t)_{t \in \R}$ be a two-sided Brownian motion, i.e., we have
$$ B_t=\left \{ \begin{array}{cc} W^{(1)}_t, & t\geq 0, \\ W^{(2)}_{-t}, & t <0, \end{array} \right. $$ 
	where $W^{(1)}= (W^{(1)}_t)_{t \geq 0}$ and $W^{(2)}= (W^{(2)}_t)_{t \geq 0}$ are two independent standard Brownian motions.
Then the process
\begin{align*} %\label{eq:fBMasMandelbrotVN}
	B_t^H = C_H\int_\R K_H(s,t)\dd B_s, \quad t \in \R,
\end{align*}
with
\begin{align*}
	C_H =  \frac{\sqrt{ \sin(\pi H) \Gamma(2H+1) }}{\Gamma(H+1/2)}
\end{align*}
and
\begin{align*}%\label{eq:kernel}
	K_H(s,t) & = \big(|t-s|^{H-1/2} -|s|^{H-1/2}\big) \mathbf{1}_{(-\infty,0)}(s)
	+ |t-s|^{H-1/2}\,  \mathbf{1}_{[0,t)}(s),
\end{align*}
defines a fBm  with Hurst parameter $H \in (0,1)$. Here, $\Gamma$ denotes the classical Gamma function. Since $x^{0}=1$, we recover in particular that
$B^{1/2}_t=B_t$, $t \in \R$. This representation of fractional Brownian motion goes back to the seminal article \cite{MandelbrotVanNess} and is called Mandelbrot--van Ness representation. Note that in \cite{MandelbrotVanNess}, a different normalization constant is used, namely
$$ C_H^{\textrm{MvN}}= \frac{1}{\Gamma(H+1/2)},$$
which leads to the variance  
$$  \mathbf{E} \left[ \left| C_H^{\textrm{MvN}} \int_\R K_H(s,t)\dd B_s \right|^2 \right]  = \frac{1}{\sin(\pi H) \Gamma(2H+1)}|t|^{2H}, \qquad t \in \R, $$ instead of 
$$   \mathbf{E}  \big[ |B_t^H|^2 \big]=|t|^{2H}, \qquad t \in \R. $$ 
See, e.g., Proposition 9.1 in \cite{taqqu}.
\smallskip

\begin{remark}
	The above extension of fBm to the time domain $\R$ instead of $[0, \infty)$ is helpful for analyzing the long-time behaviour of SDEs driven by fBm, see  Section 3.1 in \cite{Hairer} and the following Subsection \ref{subsec:ergodic}.
	\end{remark}

\smallskip

One of the main results of  \cite{KN19} is the  theorem below (Theorem 1.1 in \cite{KN19} and Theorem 2.1.1 in \cite{Koch}). Its proof relies on a stochastic Fubini theorem, see \cite{HuttonNelson}, and Kolmogorov's continuity theorem.

\begin{theorem}\label{thm:SmoothnessRegularityFBM}
	Let $k \in \mathbb{N}$ and $T >0$. Then, there exists a process $B^{H,k} = (B_t^{H,k})_{t \in [0,T]}$  such that:
	\begin{itemize}
		\item[(i)] For all $\omega \in \Omega$ the sample paths $(0,1) \times [0,T] \ni (H,t) \mapsto B_t^{H,k}(\omega) \in \mathbb{R}$ are continuous.
		\item[(ii)] For all $\omega \in \Omega$ and for any fixed $H \in (0,1)$ and $\lambda \in (0,H)$ the sample paths  $ [0,T] \ni t \mapsto B_t^{H,k}(\omega) \in \mathbb{R}$ are $ \lambda$-H\"older continuous. We even have, for all $0<a< b<1$ and $0<\gamma <a$, that there exists a non-negative random variable $K_{\gamma, a,b,k,T}>0$  such that
		\begin{align*}
			\sup_{H \in [a,b]} \big|B^{H,k}_t(\omega) - B^{H,k}_s(\omega)\big| \leq K_{\gamma, a,b,k,T}(\omega) |t-s|^\gamma, \qquad s,t \in [0,T].
		\end{align*}
		\item[(iii)]  For all $0<a< b <1$, $t \in [0,T]$ there exists $\Omega_{a,b,k,t} \in \mathcal{A}$ such that $\mathbf{P}(\Omega_{a,b,k,t})=1$ and
		\begin{align*}
			\frac{\partial^k}{\partial H^k} B_t^H(\omega) = B^{H,k}_t(\omega), \qquad H \in [a,b], \,\, \omega \in \Omega_{a,b,k,t}.
		\end{align*}
	\end{itemize}
\end{theorem}
Thus,  for $t \in [0,T]$ and almost all $\omega \in \Omega$ the map $[a,b] \ni H \mapsto B_t^{H}(\omega) \in \R$ is infinitely differentiable and the derivatives are uniformly  H\"older continuous in $t$  of all orders $\gamma <H$. 
 Moreover, the proof of the above result yields  the representation
\begin{align*} \partial_H B^H_t:= B^{H,1}_t= \big(\partial_H C_H\big) \int_\R K_H(s,t) \dd B_s + C_H \int_\R \partial_H K_H(s,t) \dd B_s, \quad t \in [0,T], \,\, a.s.,\end{align*}
for the first derivative as well as
\begin{align}\label{moment_del_B} \sup_{H \in [a,b]} \sup_{t \in [0,T]} \mathbf{E} \big[ |\partial_H B^H_t |^p \big]< \infty 
\end{align} for all $T>0$, $p \geq 1$ and $[a,b] \subset (0,1)$.  Using the process $\partial_H B^H$  one can construct a process $\widehat{B}^H$ that is indistinguishable from $B^H$ and Fr\'echet differentiable with derivative $\partial_H B^H$, see Lemma 5.1 in \cite{KN19}:

\begin{proposition}\label{MvN-Frechet} Let $H\in (0,1)$ and choose $a,b \in (0,1)$ such that $a\leq H \leq b$ and $\frac{1}{2} \in [a,b]$. Define  \begin{align*}
		\widehat{B}_t^H = B_t^{\frac{1}{2}} +  \int_{\frac{1}{2}}^H \partial_h B_t^{h}\dd h , \qquad H \in [a,b], \,\, t \in [0,T].
		\end{align*} 
	Then, we have
	\begin{itemize}
		\item[(1)]  that for fixed $H \in [a,b]$ the processes $\widehat{B}^H$ and $B^H$  are indistinguishable and 
		\item[(2)]  for all $\omega \in \Omega$ the map $ [a,b] \ni H \mapsto \widehat{B}^H(\omega) \in C([0,T];\R)$ is  Fr\'echet differentiable  with derivative $\partial_H B^H$, that is,
		\begin{align*}
			& \lim_{\delta \rightarrow 0} \frac{\sup_{t \in [0,T]} |\widehat{B}^{H+\delta}_t(\omega) - \widehat{B}^H_t(\omega) - \partial_H {B}^{H}_t(\omega)\delta |}{|\delta|}=0.
		\end{align*}
	\end{itemize}
	
\end{proposition}
In the following, we will identify the fBms $B^{H}$ and $\widehat{B}^H$ and  we will refer to them as Mandelbrot--van Ness fractional Brownian motion (MvN-fBm).

\smallskip

The $H$-dependence of the MvN-fBm has been also analysed in \cite{HRtheo}, but with a particular focus on the long-time dependence. In  Theorem 3.1 of \cite{HRtheo} the authors obtain the following result:

\begin{theorem}\label{thm:whole-regularity}
	Let $[a,b] \subset (0,1)$.
	Then, for any $\varepsilon \in (0,1)$ and any $p\geq 1$,
	there exists a non-negative random variable $ K_{a,b,\varepsilon,p}$ with  $ \mathbf{E}[|K_{a,b,\varepsilon,p}|^p]<\infty$
	such that 
	\begin{align*}
		| B_t^H - B_{t'}^{H'} |  \leq K_{a,b,\varepsilon,p}\, (1+t')^{2\varepsilon a+ b}  \left(  \min \{ 1,  |t-t'|^{a} \} + |H-H'| \right)^{1-\varepsilon}, \,\, \, t' \ge t \ge 0, \,\, H,H' \in [a,b],
	\end{align*} almost surely.
	\end{theorem}
The proof of this result relies on mean square estimates for the auxiliary  two-parameter process $\mathbb{B}$ defined as
\begin{align*}%
	\mathbb{B}_t^H = (1+t)^{-(b+\varepsilon)} \, B_t^H , \quad  t\geq 0, \, \, H\in [a,b] ,
\end{align*}
 Kolmogorov's continuity theorem and  a multi-parameter version of the Garsia--Rodemich--Rumsey (GRR) lemma, see Lemma 3.6 in \cite{HRtheo}.

\begin{remark} In fact, the authors use in \cite{HRtheo} the original representation from \cite{MandelbrotVanNess} with the constant $C^{\operatorname{MvN}}$. Since $[a,b] \ni H \mapsto \sin(\pi H) \Gamma(2H+1) \in (0,\infty)$ is Lipschitz continuous, this does not affect the above result. 
\end{remark}

\smallskip
\medskip

\section{The dependence of SDEs driven by fBm on the Hurst parameter}\label{sec:SDE}

Now we turn to the analysis of SDEs driven by fBm.

\subsection{SDEs  with additive noise}
For additive noise, i.e., $\sigma(x)= \Sigma $, $x \in \R^{d}$, with a given fixed matrix $\Sigma \in \mathbb{R}^{d \times m}$,  equation \eqref{det_int} is just an  ordinary differential equation. Indeed, setting $z_t=x_t-\Sigma g_t$, we have the dynamics
	\begin{align*}%\label{det_int_2}
	\dd z_t & = \mu (z_t + \Sigma g_t) \dd t, \quad t \geq 0, \qquad z_0 = x_0.
	\end{align*}
Since $g$ is continuous, for existence of a unique solution it is sufficient that $\mu$ is globally Lipschitz continuous due to the classical Picard--Lindel\"of Theorem (\cite{picard,lindeloef}).
Moreover, for the SDE
\begin{align*}
	X_t^H = x_0+  \int_0^t \mu(X_s^H) \dd s + \Sigma B_t^H, \quad t \geq 0,
\end{align*}
we immediately obtain the following result. (Recall that for $B^H$ we are using the MvN-fBm from Proposition \ref{MvN-Frechet}.)

\begin{lemma} Let  $[a,b] \subset (0,1)$, $T >0$ and let $ \mu \colon \R^d \rightarrow \R$ be globally Lipschitz continuous. Then, there exists a  non-negative random variable $K_{a,b,T}$ such that
	$$ \sup_{t \in [0,T]} \|X_t^{H_1}-X_t^{H_2}\| \leq K_{a,b,T} |H_1-H_2|, \qquad H_1,H_2 \in [a,b]. $$
	\end{lemma}
\begin{proof}
An application of Gronwall's lemma  gives that
$$ \sup_{t \in [0,T]} \|X_t^{H_1}-X_t^{H_2}\| \leq  \| \Sigma\| \,  \exp(K_{\mu} T) 	\sup_{t \in [0,T]} \|B_t^{H_1}-B_t^{H_2}\| $$
for all $T> 0$.  Here, $K_{\mu}$ is the Lipschitz constant of $\mu$. By Proposition \ref{MvN-Frechet} we have that
$$ \sup_{t \in [0,T]} \|B_t^{H_1}-B_t^{H_2}\| \leq |H_1-H_2| \sup_{H \in [a,b]} \sup_{t \in [0,T]} \| \partial_H B_t^{H}\|.$$
Note that  $\sup_{H \in [a,b]} \sup_{t \in [0,T]} \| \partial_H B_t^{H}\|$ is finite due to Theorem \ref{thm:SmoothnessRegularityFBM} (i). Thus, the assertion follows with
$$  K_{a,b,T}= \| \Sigma\| \,  \exp(K_{\mu} T) \sup_{H \in [a,b]} \sup_{t \in [0,T]} \| \partial_H B_t^{H}\|. $$
\end{proof}
Thus, SDEs with additive fractional noise given by the MvN-fBm are pathwise Lipschitz continuous (on a finite time interval) with respect to the Hurst parameter.	
The long-time dependence on $H$ of ergodic SDEs with additive noise is the main focus of  \cite{HRtheo}, which we will discuss in the next subsection.

	\smallskip

\subsection{Ergodic SDEs with additive noise} \label{subsec:ergodic}
Let us consider again the case of additive noise, but now with a dissipative drift. Thus, we consider the SDE
\begin{align}\label{diss_SDE}
	X_t^H = x_0+  \int_0^t \mu(X_s^H) \dd s +  B_t^H, \quad t \geq 0,
\end{align}
where $m=d$ and the  drift $\mu \in C^1(\R^d;\R^d)$ is  globally Lipschitz continuous and dissipative, i.e., there exist constants $K_{\mu},\kappa_{\mu}>0$ such that
\begin{align} \label{lip} \tag{D1}
	\| \mu(x)- \mu(y) \| & \leq K_{\mu} \|x-y\|, \qquad 	\,\, \,\, \, \, x,y  \in \R^d , \\
	\langle x -y, \mu(x)- \mu(y) \rangle  & \leq - \kappa_{\mu} \|x-y\|^2 , \qquad x,y  \in \R^d.  \tag{D2} \label{dissip}
\end{align}
Under these conditions the moments of $X^H_t$ are uniformly bounded for $t \geq 0$, i.e., we have
\begin{align} \label{diss_moment_bound} \sup_{t \geq 0} \mathbf{E} \big[ \|X_t^H\|^p \big] < \infty \end{align}
for all $p \geq 1$, see, e.g., Proposition 2.2 in \cite{NTsisp}.
Moreover, it is well-known, see, e.g., \cite{Hairer,GKN}, that  the solution of SDE \eqref{diss_SDE} converges to a stationary solution for $t \rightarrow \infty$ and  is ergodic.
For the following result see, e.g., Theorem 2.1 and Proposition 2.3 in \cite{NTsisp}.  
\begin{proposition}\label{stat_ergodic} Assume that $\mu  \in C^1(\R^d;\R^d)$ satisfies \eqref{lip} and \eqref{dissip}. Then, there exists a stochastic process $\overline{X}_t^H, t \geq 0$, such that 
	\begin{itemize}
		\item[(a)] we have  ${\overline{X}_t^H} \stackrel{\mathcal{L}}{=} {\overline{X}_0^H}$ for all $t \geq 0$ and  $\mathbf{E}\big[\| \overline{X}_0^H\|^p \big]<\infty$ for all $p \geq 1$,
		\item[(b)] we have   $$\lim_{t \rightarrow \infty}    \|{X}_t^H- \overline{X}_t^H\|\, \stackrel{a.s.}{=} \, 0,$$ 
		\item[(c)]  we have  $$\lim_{T \rightarrow \infty} \frac{1}{T} \int_0^T \varphi(X^H_t) \dd t  \, \stackrel{a.s.}{=} \, \mathbf{E} \big[ \varphi( \overline{X}^H_0) \big] $$  for all $\varphi \in C^{1}_{pol}(\mathbb{R}^d; \mathbb{R})$.
	\end{itemize}
	
\end{proposition}

In view of this result it is  natural  to analyse the long-time behaviour of the solution $X^H$ and as well of the stationary solution $\overline{X}^H$ with respect to $H$. For the scalar linear SDE
\begin{align*}%\label{diss_OUP}
	U_t^H = x_0 - \kappa  \int_0^t U_s^H \dd s +  B_t^H, \quad t \geq 0,
\end{align*}
with $\kappa >0$ its  solution is called fractional Ornstein--Uhlenbeck process and the stationary solution is given by
\begin{align*}%\label{diss_OUP}
	\overline{U}_t^H =       \int_{-\infty}^t  \exp(-\kappa(t-s)) \dd {B}^{H}_{s}, \qquad t \in \R.
\end{align*}
For the stationary fractional Ornstein--Uhlenbeck process one has the following result, see Proposition 4.2 in \cite{HRtheo}: \begin{proposition}\label{thm:regularity-OU}
	Let $0<a<b<1$ and $\kappa=1$. For any $\varepsilon \in (0,1)$ and $p\ge1$, there exists a non-negative random variable $K_{a,b,\varepsilon,p}$ with  $\mathbf{E}[|K_{a,b,\varepsilon,p}|^p]<\infty$ such that we have
	\begin{align*}
		|\overline{U}_t^{H} - \overline{U}_{t'}^{H'}| \leq K_{a,b,\varepsilon,p} \, (1+t')^{\varepsilon} \left( \min \{1 , |t'-t|^{a} \} +|H-H'|\right)^{1-\varepsilon}, \qquad t' \ge t \ge 0, \, \, H,H' \in [a,b],
	\end{align*}
	almost surely.
\end{proposition}
The proof of this result is similar to Theorem \ref{thm:whole-regularity} and relies again on a rescaling of the fractional Ornstein--Uhlenbeck process, analysing the increments of the arising two-parameter process and an application of a multi-dimensional GRR lemma.
By comparing the solution of the general SDE \eqref{diss_SDE} with the stationary fractional Ornstein--Uhlenbeck process the following result (see Theorem 4.3 in \cite{HRtheo}) is obtained:

\begin{theorem}\label{thm:regularity-SDE}
		Let $0<a<b<1$. Assume that $\mu  \in C^1(\R^d;\R^d)$ satisfies \eqref{lip} and \eqref{dissip}.
\begin{itemize}
		\item[(i)]   For any $\varepsilon \in (0,1)$ and $p\ge1$, there exists a non-negative random variable $K_{a,b,\varepsilon,p}$ with  $\mathbf{E}[|K_{a,b,\varepsilon,p}|^p]<\infty$ such that we have
		\begin{align*}
			\|X_t^{H} - X_{t}^{H'}\| \leq K_{a,b,\varepsilon,p}\, (1+t)^{\varepsilon}\, |H-H'|^{1-\varepsilon}, \qquad  t \ge 0, \, \, H,H' \in [a,b],
			\end{align*}
			almost surely.
		\item[(ii)] Let $p \ge 1$. There exists a constant $C_{a,b,p}>0$ such that 
		\begin{align*}
		\mathbf{E} \big [ \big\|X_t^{H} - X_t^{H'}\big\|^p \big ]  \leq C_{a,b,p}\, |H-H'|^{p} , \qquad   t \ge 0, \,\,  H,H' \in [a,b].
		\end{align*}
	\end{itemize}
\end{theorem}
Exploiting that the variance of  the ergodic means
decreases over time  one can obtain also the following result (Theorem 4.6 in \cite{HRtheo}):
\begin{theorem}\label{th:ergodic-OU} 	 Assume that $\mu  \in C^1(\R^d;\R^d)$ satisfies \eqref{lip} and \eqref{dissip}. Let $0<a<b<1$,  $ \beta \in  (0,1)$ and $p \ge 1$. Then, there exists a non-negative random variable $K_{a,b,\beta,p}$ with  $\mathbf{E}[|K_{a,b,\beta,p}|^p]<\infty$ such that we have 
	\begin{align*}%
		\frac{1}{1+t} \int_0^{1+t} \| X_s^{H}-X_s^{H'}\|^2 \dd s  \leq K_{a,b,\beta,p} \,\,   |H-H'|^{\beta},  \qquad t \ge 0, \, \, H,H' \in [a,b],
	\end{align*} almost surely.
\end{theorem}
This result is of  importance for statistical problems involving  SDEs with additive fractional noise, namely for analysing  estimators which simultaneously estimate $H$ and other parameters of the SDE, see \cite{HRstat}.

The law of the stationary fractional Ornstein--Uhlenbeck process depends in a smooth way on $H$, since $\overline{U}^H$ is a Gaussian process which satisfies in particular 
$$ \mathbf{E} \big[  \overline{U}^H_0 \big] =0 \qquad \textrm{and} \qquad   \mathbf{E} \big[ \big|  \overline{U}^H_0 \big|^2 \big] =     \frac{ \Gamma(2H+1)}{2\kappa^{2H}}. $$
 See, e.g., the proof of Proposition 3.12 of \cite{Hairer} or Proposition 3.1 in \cite{Sch-diss}. An immediate consequence of Theorem \ref{thm:regularity-SDE} is that the law of the stationary solution $\overline{X}^H$ is also Lipschitz continuous with respect to $H$:

\begin{proposition}
	Let $0<a<b<1$. Assume that $\mu \in C^1(\R^d;\R^d)$ satisfies \eqref{lip} and \eqref{dissip}. Then, for 
any $\varphi \in C_{\textrm{pol}}^1(\R^d;\R)$  the map
$$ \mathcal{E}_{\varphi,a,b} \colon [a,b] \rightarrow \R, \qquad \mathcal{E}_{\varphi,a,b}(H)= \mathbf{E} \big[ \varphi(\overline{X}_0^H) \big],$$
is globally Lipschitz continuous.
\end{proposition}
	
	\begin{proof} First note that   equation \eqref{diss_moment_bound}  and Theorem \ref{thm:regularity-SDE} (ii) imply the existence of a constant $C_{a,b,p}^*>0$ such that
		\begin{align}
			 \label{diss_moment_bound_uniform} \sup_{t \geq 0} \sup_{H \in [a,b]} \mathbf{E} \big [ \|X_t^H\|^p \big]\leq C_{a,b,p}^*
		\end{align}
		for all $p \geq 1$.  Property (c) of Proposition \ref{stat_ergodic} yields that
	\begin{align*}   \mathbf{E} \big[ \varphi( \overline{X}^{H_1}_0) \big] -  \mathbf{E} \big[ \varphi( \overline{X}^{H_2}_0) \big] \, & = \, \mathbf{E} \left[ \liminf_{T \rightarrow \infty} \frac{1}{T} \int_0^T  \big( \varphi(X^{H_1}_t) -  \varphi(X^{H_2}_t) \big)  \dd t \right]  %\\ &  = \, \mathbf{E} \left[ \liminf_{T \rightarrow \infty} \frac{1}{T} \int_0^T  \big( \varphi(X^{H_1}_t) -  \varphi(X^{H_2}_t) \big)  \dd t \right] 
	  \end{align*}
	for $H_1,H_2 \in [a,b]$. Since $\varphi \in C_{{pol}}^1(\R^d;\R)$, there exists $p \geq 1$ and $K_{\varphi}>0$ such that
		$$ | \varphi(x)- \varphi(y) | \leq K_{\varphi} \left( 1+ \|x\|^p+\|y\|^p\right) \|x-y\|, \qquad x,y \in \R^d. $$	
		Thus,  we obtain
	$$   \left| \mathbf{E} \big[ \varphi( \overline{X}^{H_1}_0) \big] -  \mathbf{E} \big[ \varphi( \overline{X}^{H_2}_0) \big] \right| \leq  {K_{\varphi}}  \, \mathbf{E} \left[  \liminf_{T \rightarrow \infty} \frac{1}{T} \int_0^T ( 1+ \|X^{H_1}_t\|^p+\|X^{H_2}_t\|^p )  \| X^{H_1}_t -  X^{H_2}_t \|  \dd t  \right] $$ 
	and Fatou's lemma  implies 
		\begin{align*} & \left| \mathbf{E} \big[ \varphi( \overline{X}^{H_1}_0) \big] -  \mathbf{E} \big[ \varphi( \overline{X}^{H_2}_0) \big] \right| %\leq & \mathbf{E} \left[  \lim_{T \rightarrow \infty} \frac{1}{T} \int_0^T  \big | \varphi(X^{H_1}_t) -  \varphi(X^{H_2}_t) \big|  \dd t  \right] \\ & \qquad  \qquad 
 \leq  K_{\varphi}    \liminf_{T \rightarrow \infty} \frac{1}{T} \int_0^T  \mathbf{E}  \left[  ( 1+ \|X^{H_1}_t\|^p+\|X^{H_2}_t\|^p )\| X^{H_1}_t -  X^{H_2}_t \|   \right] \dd t.  \end{align*}
		Using the Cauchy-Schwarz inequality and $(a+b+c)^2 \leq 4(a^2+b^2+c^2)$, we  obtain 
		\begin{align*} & 
\left| \mathbf{E} \big[ \varphi( \overline{X}^{H_1}_0) \big] -  \mathbf{E} \big[ \varphi( \overline{X}^{H_2}_0) \big] \right| \\ & \qquad  \qquad  \leq  2K_{\varphi}   \liminf_{T \rightarrow \infty} \frac{1}{T} \int_0^T  \left( \mathbf{E}  \left[  1+ \|X^{H_1}_t\|^{2p}+\|X^{H_2}_t\|^{2p} \right] \right)^{1/2} \left( \mathbf{E} \left[ \| X^{H_1}_t -  X^{H_2}_t \big\|^2   \right] \right)^{1/2}\dd t  \end{align*}
		and Theorem \ref{thm:regularity-SDE} (ii) then gives
			\begin{align*} &
\left| \mathbf{E} \big[ \varphi( \overline{X}^{H_1}_0) \big] -  \mathbf{E} \big[ \varphi( \overline{X}^{H_2}_0) \big] \right|  \\ & \qquad  \qquad  \leq  2 K_{\varphi} \sqrt{C_{a,b,2} } |H_1-H_2| \left(  \liminf_{T \rightarrow \infty} \frac{1}{T} \int_0^T  \left( \mathbf{E}  \left[  1+ \|X^{H_1}_t\|^{2p}+\|X^{H_2}_t\|^{2p} \right] \right)^{1/2} \dd t \right) .  \end{align*}
		By equation  \eqref{diss_moment_bound_uniform} we have
		$$  \liminf_{T \rightarrow \infty} \frac{1}{T} \int_0^T  \left( \mathbf{E}  \left[  1+ \|X^{H_1}_t\|^{2p}+\|X^{H_2}_t\|^{2p} \right] \right)^{1/2} \dd t  \leq \sqrt{1+2C^*_{a,b,2p}} $$
for all $H_1,H_2 \in [a,b]$ and   the assertion now follows. 
	\end{proof}

\smallskip
 
 \subsection{The Doss--Sussmann approach} For $m=1$, the so-called Doss--Sussmann approach from \cite{Doss,Sussmann} is a strikingly simple and elegant concept to give meaning to the object
 \begin{align}\label{suss_ode}
 	\dd x(t)  = \mu(x(t)) \dd t +  \sigma(x(t))\dd g_t, \quad  t \in [0,T], \qquad x(0)=x_0 \in \mathbb{R}^d,
 \end{align}	
 with   $g \in C([0,T];\mathbb{R})$.
 
 Namely, a function $\gamma  \in C([0,T];\mathbb{R}^d)$ is called a solution to this equation,
 \begin{itemize}
 	\item[(i)]
 	if there exists a continuous map
 	$\Gamma \colon C([0,T];\mathbb{R}) \rightarrow C([0,T];\mathbb{R}^d)$ such that, for every $v \in C^1([0,T];\mathbb{R})$, $\Gamma(v)$ is a classical solution of the ODE
 	$$ x'(t)= \mu(x(t))+\sigma(x(t))v'_t, \quad  t \in [0,T], \qquad x(0)=x_0, $$
 	\item[(ii)] and $\gamma=\Gamma(g)$.
 \end{itemize}	
 In particular, if $\mu$ and $\sigma$ are globally Lipschitz continuous, then the differential equation \eqref{suss_ode} has a unique solution in the above sense, see \cite{Doss,Sussmann}. 
 
 Consequently, we can view 
 $$ X^H(\omega)=\Gamma(B^H(\omega)), \qquad \omega \in \Omega, $$
 as the unique pathwise solution of 
 	\begin{align*}
 	X_t^H = x_0+  \int_0^t \mu(X^H_s) \dd s +  \int_0^t \sigma(X^H_s)\dd B^H_s, \qquad t \in [0,T].
 \end{align*}
 
For $d=m=1$, we have even a more explicit representation of $\Gamma$ under slightly stronger assumptions on the coefficients, see, e.g., \cite{Doss}. So, let  $\mu \in   C^1(\R;\R)$ and $\sigma \in C^2(\R;\R)$ such that $\mu'$ and $\sigma'$ are bounded.
 Further, let $h \colon \R \times \R \to \R$ be the unique solution of
 \begin{align}\label{eq:Diff_h}
 	\frac{\partial h}{\partial \beta}(\alpha,\beta) = \sigma(h(\alpha,\beta)), \qquad h(\alpha, 0) = \alpha,
 \end{align}
 and for a given $g \in C([0,T]; \mathbb{R})$, let $D\in C^1([0,T];\mathbb{R})$ be the solution of the ODE
 \begin{align*}%\label{eq:DEDossD}
 	D'(t) = f(D(t),g_t),  \quad  t \in [0,T],\qquad D(0) = x_0,
 \end{align*}
 with $f :\R \times \mathbb{R} \rightarrow \mathbb{R}$  given by
 \begin{align*} %\label{eq:fDef}
 	f(x,y) = \exp\bigg(- \int_0^{y} \sigma'(h(x,s))\dd s\bigg) \mu \big(h(x,y)\big).
 \end{align*}
 Then, the unique Doss--Sussmann solution  to equation \eqref{suss_ode} can be written as
 \begin{align} \label{suss_gamma} \Gamma(g)(t)=h(D(t),g_t) \end{align} and the map $\Gamma$ is locally Lipschitz continuous. This smoothness result can  be extended to Fr\'echet differentiability:
 
 \begin{lemma}
Let $\mu \in   C^1(\R;\R)$ and $\sigma \in C^2(\R;\R)$ such that $\mu'$ and $\sigma'$ are bounded. Then, the map 
$ \Gamma \colon C([0,T]; \R) \rightarrow  C([0,T]; \R) $ defined by \eqref{suss_gamma}  is Fr\'echet differentiable.
 \end{lemma}
 \begin{proof}
 	This follows from Lemma 4.1 in \cite{KN19} and the smoothness of $h$ as the solution of the partial differential equation \eqref{eq:Diff_h}.
 \end{proof}
 
Since $B^H$ is given by the Mandelbrot--van Ness fBm from Proposition \ref{MvN-Frechet}, we  obtain by the chain rule, the above lemma and Proposition \ref{MvN-Frechet} for any fixed $\omega \in \Omega$ and any  $t>0$  that the map
 $$ \mathcal{S}_{t,(0,1)}\colon \Omega \times (0,1) \rightarrow \R, \qquad \mathcal{S}_{t,(0,1)}(\omega,H)=X^H_t(\omega),$$
 is differentiable with respect to $H$.

 In several cases one can derive explicit representations for $Y^H=\partial_H X^H$.
 
 \begin{itemize}
\item[(i)]  For the linear equation
 $$ \dd X_t^H= \alpha X_t^H \dd t + \beta X_t^H \dd B^H_t$$ with $\alpha, \beta \in \mathbb{R}$, we have
 $$ X_t^H=x_0 \exp\left( \alpha t + \beta B_{t}^H \right)$$ and
 $$ Y_t^H = \beta  X_t^H \, {\partial_H}B_{t}^H.$$

 \item[(ii)]

 In the case of additive noise, i.e., $\sigma(x)=1$ for all $x \in \mathbb{R}$, the derivative satisfies
 $$ Y_t^H= \int_0^t \mu'(X_{\tau}^H) Y_{\tau}^H \dd \tau + \partial_H B_t^H, $$ 
 and therefore we have
 \begin{align*}
 	Y_t^H &= \int_0^t\exp\left( \int_s^t \mu'(X_{\tau}^H) \dd \tau \right) \dd ( {\partial_H}B_{s}^H),
 \end{align*} 
 using integration by parts and  $\partial_HB^H_0 = 0$. Note that the boundedness of $\mu'$, Gronwall's lemma and equation \eqref{moment_del_B} imply that
 \begin{align} \label{bound_mom_frechet_add_noise}
 \sup_{H \in [a,b]}	\sup_{t \in [0,T]}	\mathbf{E} \big [	|Y_t^H|^p \big ]< \infty
 \end{align} for all $T>0$, $p \geq 1$ and $[a,b] \subset (0,1)$.
 \item[(iii)]
 For non-additive noise, i.e., $\sigma' \neq 0$, one expects to obtain the representation
 $$ Y_t^H= \int_0^t \exp \left(\int_s^t \mu'(X_{\tau}^H) \dd \tau + \int_s^{t} \sigma'(X_{\tau}^H) \dd B_{\tau}^H \right) \sigma(X_s^H)  \dd   ( {\partial_H} B_s^H ).$$
 For $H>1/2$ this is indeed the case, see Subsection \ref{sec:Hge1/2}. However, for $H\leq 1/2$, a meaningful interpretation of even the simpler object
 $$ \int_0^T  B_s^H \dd   ( {\partial_H} B_s^H)$$
 is an open question, see  Subsection \ref{sec:Hleq1/2}.
 
 \end{itemize}
 	 Under additional assumptions one can establish the following result for the marginal distributions, see Proposition 4.1 in \cite{RichardTalay},  by using Malliavin techniques and a combined Lamperti-parabolic PDE transformation.

 \begin{proposition}\label{lip_exp_d=1}
 Let $1/4<a<b<1$,
 $\mu \in C^{1}_b(\R; \R)$, $\sigma \in C^2_b(\R ; \R)$ and $\inf_{x \in \R} \sigma(x) >0$. Moreover, let $\varphi \in C_b^{2}(\R; \R)$ with $\varphi'' \in C^{\lambda}(\R;\R)$ for some $\lambda > 0$.
 Then, for all $t \geq 0$, the maps   $$ \mathcal{E}_{\varphi,a,b,t}\colon [a,b] \rightarrow \R, \qquad \mathcal{E}_{\varphi,a,b,t}(H)= \mathbf{E} \big[ \varphi(X_t^H) \big],$$
 are  globally Lipschitz continuous.
  \end{proposition}

The  Lamperti-transformation
$$F(x)= \int_0^x \frac{1}{\sigma(y)}\dd y, \qquad x \in \mathbb{R}, $$
transforms the SDE
\begin{align*}
	X_t^H = x_0+  \int_0^t \mu(X^H_s) \dd s +  \int_0^t \sigma(X^H_s) \dd B^H_s, \qquad t \in [0,T],
\end{align*} via $Z_t^H=F(X_t^H)$
into the additive noise SDE 
$$   Z_t^H=F(x_0) + \int_0^t \frac{\mu(F^{-1}(Z_s^H))}{\sigma(F^{-1}(Z_s^H)) } \dd s + B_t^H, \qquad t \in [0,T].   $$
Note that the assumptions on the coefficients imply that $F$ as well as $F^{-1}$ are globally Lipschitz continuous and also that the new drift
$$ \mu_{F}\colon  \R  \rightarrow \R, \quad \mu_{F}(x)= \frac{\mu(F^{-1}(x))}{\sigma(F^{-1}(x))},$$ belongs to $C_b^{1}(\R;\R).$ Hence, the Fr\'echet differentiability of the Doss--Sussmann map and the representation of the derivative as well as equation \eqref{bound_mom_frechet_add_noise} provide an alternative proof for the above statement: we have
\begin{align*}
\left|	\mathbf{E} \big[ \varphi(X_t^{H_1}) \big] - 	\mathbf{E} \big[ \varphi(X_t^{H_2}) \big] \right| &= \left|	\mathbf{E} \big[ \varphi(F^{-1}(Z_t^{H_1})) \big] - 	\mathbf{E} \big[ \varphi(F^{-1}(Z_t^{H_2})) \big]	\right| \\& \leq K_{\varphi} K_{F^{-1}}  	\mathbf{E} \big [ \big|Z_t^{H_1}- Z_t^{H_2} \big|   \big] \\ & = K_{\varphi} K_{F^{-1}} 
	\mathbf{E} \left[  \left|  \int_{H_1}^{H_2} Y_t^h \dd h   \right|  \right]  \\ & \leq  K_{\varphi} K_{F^{-1}} 
 \int_{H_1}^{H_2} 	\mathbf{E} \big[  |  Y_t^h |  \big] \dd h   \\ & \leq   \left( K_{\varphi} K_{F^{-1}} \sup_{H \in [a,b]}	\sup_{t \in [0,T]}	\mathbf{E} 	\big[ |Y_t^H| \big] \right) \cdot |H_1-H_2|.
\end{align*}

Using the Lamperti-transformation also a more  irregular functional of the solution is analysed   in  \cite{RichardTalay:prem,RichardTalay}, namely the 
first hitting time of the point $x=1$, i.e.,
of $$ \tau_X^H = \inf \{ t \geq 0: \, X_t^H=1\}. $$
Applying again PDE as well as Malliavin  techniques, the authors
  obtain in Theorem 5.2 of \cite{RichardTalay} the following result:
 
 \begin{theorem}
  Let $H \in (1/4,1)$,  $\lambda >0$ and $x_0<1$. Moreover, let 
 $\mu \in C^{1}_b(\R; \R)$, $\sigma \in C^2_b(\R; \R)$ and $\inf_{x \in \R} \sigma(x) >0$. Then, there exists a constant $C(\lambda,H,x_0,\mu,\sigma)>0$ such that
 $$ \left| \mathbf{E} \big[ \exp({-\lambda \tau_X^H}) \big] -   \mathbf{E} \big[\exp({-\lambda \tau_X^{ 1/2 }} ) \big] \right| \leq C(\lambda,H,x_0,\mu,\sigma) \left| H  - \tfrac{1}{2} \right| $$
 for all $$ \lambda > \sup_{x \in \mathbb{R}} \left|  \mu_F'(x)\right|.$$
 \end{theorem}
 For an explicit representation  of the constant  $C(\lambda,H,x_0,\mu,\sigma)$ and a detailed discussion of its  properties, see Theorem 5.2 and Remark 5.5 in \cite{RichardTalay}.

\smallskip
 
 \subsection{The multi-dimensional case for $H>1/2$}\label{sec:Hge1/2}
 
 Multi-dimensional SDEs driven by a multi-dimensional fBm are again understood in a pathwise sense. For $H>1/2$ one can rely on Young integration theory,
 since for $f \in C^{\alpha}([0,T]; \R)$, $g  \in C^{\beta}([0,T]; \R)$ with $\alpha +\beta >1$ the Riemann--Stieltjes integral 
 $$ \mathcal{I}(f,g)=\int_0^T f(t) \dd g(t) $$
 exists, see, e.g., the seminal article \cite{Young} or the work \cite{zaehle}, which relies on fractional calculus.
 
 For our purposes, it will be beneficial to use the approach of the authors of  \cite{NR}, who work in Besov-type spaces due to their use of fractional calculus. So, let $\alpha \in (0,1/2)$ and denote by $W_1^\alpha([0,T];\R^d)$ the space of measurable functions $f \colon [0,T] \to \R^d$ such that
 \begin{align*}
 	\|f\|_{\alpha,1} = \sup_{t \in [0,T]} \bigg(\|f(t)\| + \int_0^t \frac{\|f(t)-f(s)\|}{|t-s|^{1+\alpha}} \dd s\bigg) < \infty.
 \end{align*}
 Moreover, denote by $W_{2}^{1-\alpha}([0,T];\R^m)$ the set of measurable functions $g \colon [0,T] \to \R^m$ such that
 \begin{align*}
 	\|g\|_{1-\alpha,2} := \sup_{0\leq s < t \leq T} \bigg(\frac{\|g(t)-g(s)\|}{|t-s|^{1-\alpha}} + \int_s^t \frac{\|g(y)-g(s)\|}{|y-s|^{2-\alpha}} \dd y\bigg) < \infty.
 \end{align*}
 
Note that one has the embeddings 
 \begin{align*}
 	C^{\alpha+\varepsilon}([0,T];\R^d) \subseteq W_1^\alpha([0,T];\R^d)
 \end{align*}
 and
 \begin{align*}
 	C^{1-\alpha+\varepsilon}([0,T];\R^d) \subseteq W_2^{1-\alpha}([0,T];\R^d) \subseteq C^{1-\alpha}([0,T];\R^d),
 \end{align*}
 for $\varepsilon >0$.
 Thus, the sample paths of the MvN-fBm from Proposition \ref{MvN-Frechet} belong  to $W_1^{\alpha}([0,T];\R^d)$ for $\alpha <H$ and to $W_2^{1-\alpha}([0,T];\R^d)$ for $\alpha > 1-H$.

 While  \cite{NR} establishes existence and uniqueness for the deterministic integral equation
 \begin{align*}
 	x_t = x_0 + \int_0^t \mu(x_s) \dd s + \int_0^t \sigma(x_s) \dd g_s, \qquad t \in [0,T],
 \end{align*}
 for $g \in W_{2}^{1-\alpha}([0,T];\R^m)$,  $x_0 \in \R^d$ and $\mu \colon \R^d \rightarrow \R^d$, $\sigma\colon \R^d \rightarrow \R^{d \times m}$ satisfying mild smoothness assumptions, the work \cite{NualartSauss} analyses the Fr\'echet differentiability of the solution map and obtains the following result:
 \begin{proposition}
 	Let $\alpha \in (0,\frac{1}{2})$ and $g\in W_2^{1-\alpha}([0,T];\R^m)$. Moreover, let $\mu \in C_b^{3}(\R^d;\R^d)$ and $\sigma \in C_b^{3}(\R^d;\R^{d \times m})$. Denote by $x \in W^{\alpha}_1([0,T];\R^d)$ the solution of 
 	\begin{align*}
 		x_t = x_0 + \int_0^t \mu(x_s) \dd s + \int_0^t \sigma(x_s) \dd g_s, \qquad t \in [0,T].
 	\end{align*}
 	The mapping $$\Gamma \colon W^{1-\alpha}_2([0,T];\R^m) \to W^\alpha_1([0,T];\R^d), \,\, \ g \mapsto x(g),$$ is Fréchet differentiable.
 	For $h \in W^{1-\alpha}_2([0,T];\R^m)$ its derivative is given by
 	\begin{align*}
 		\big(\Gamma'(g)h\big)(t) = \int_0^t \Phi_t(s) \dd h_s,
 	\end{align*}
 	where $\Phi_t(s) \in \R^{d\times m}$ is defined as follows: letting $\partial_k$ denote the partial derivative with respect to the $k$-th variable, $s \mapsto \Phi_t(s)$ satisfies
 	\begin{align*}
 		\Phi^{ij}_t(s) = \sigma^{ij}(x_s) + \sum_{k = 1}^d \int_s^t \partial_k \mu^i(x_u) \Phi^{kj}_u(s) \dd u + \sum_{k =1}^{d} \sum_{l=1}^m \int_{s}^t \partial_k \sigma^{il}(x_u) \Phi^{kj}_u(s) \dd g_u^l
 	\end{align*}
 	for $ 0 \leq s \leq t \leq T$ and $\Phi^{ij}_t(s) = 0$ for $s>t$, where $i = 1, \dots,d, \ j = 1, \dots,m$.
 \end{proposition}

 It turns out that the MvN-fBm from Proposition \ref{MvN-Frechet} is also Fr\'echet differentiable with  $W^{1-\alpha}_2([0,T];\R^m)$ as target space, see page 30 in \cite{Koch}.
 
 \begin{proposition}
 		Let  $0 <a < b <1$, $1/2< 1-\alpha < a < H < b$ and
 		\begin{align*}
 			\widehat{B}_t^H := B_t^{1/2} + \int_{1/2}^H \partial_h B_t^{h} \dd h, \qquad H \in [a,b], \quad t \in [0,T].
 		\end{align*}
 		Then 
 	%	\smallskip
 	%	\begin{itemize}
 	%		\item[(i)] for fixed $H\in [a,b]$ the processes $\widehat{B}^H$ and $B^H$ are indistinguishable
 	%		\smallskip
 		for all $\omega \in \Omega$ the map \begin{center} $ [a,b] \ni H \mapsto \widehat{B}^H(\omega) \in W_2^{1-\alpha}([0,T];\R)$ \end{center} is  Fr\'echet differentiable with derivative $\partial_H B^{H}$. 
 	
\end{proposition}

As before, we identify $\widehat{B}^H$ and $B^H$. Thus, we obtain for  all $\omega \in \Omega$, by the chain rule, that
 \begin{align}\label{deriv_multi-sde_1}
 	{\partial_ H} X^H_t(\omega) = {\partial_H} \big( \Gamma(B^H(\omega))(t) \big) = \Big(\Gamma'(B^H(\omega))  \partial_H B^{H}(\omega)\Big)(t) =\int_0^t \Phi_t(s) \dd \partial_H B^{H}_s(\omega) ,
 \end{align}
 with $X^H = \Gamma(B^H)$ and where $\Phi_t(s)$  is given by \begin{equation}\label{deriv_multi-sde_2}
 \begin{aligned}
 	\Phi^{ij}_t(s) = \sigma^{ij}\big(X^H_s(\omega)\big) &+ \sum_{k = 1}^d \int_s^t \partial_k \mu^i\big(X^H_u(\omega)\big) \Phi^{kj}_u(s) \dd u\\
 	&+ \sum_{k =1}^{d} \sum_{l=1}^m \int_{s}^t \partial_k \sigma^{il}\big(X^H_u(\omega)\big) \Phi^{kj}_u(s) \dd B^{H,l}_u(\omega)
 \end{aligned} \end{equation}
 for $ 0 \leq s \leq t \leq T$ and $\Phi^{ij}_t(s) = 0$ for $s>t$, where $i = 1, \dots,d, \ j = 1, \dots,m$. 
 In the one-dimensional case we have in particular that
 \begin{align*}
 	\partial_H X_t^H 
 	& = \int_0^t\exp\left( \int_s^t \mu'(X_{\tau}^H) \dd \tau +  \int_s^t \sigma'(X_{\tau}^H) \dd B_{\tau}^H \right)  \sigma(X_s^H) \dd \left( {\partial_H}B_{s}^H\right), \qquad   t \in [0,T].
 \end{align*}

 \smallskip

 \subsection{The multi-dimensional case for $H\leq 1/2$} \label{sec:Hleq1/2}
General multi-dimensional SDEs driven by fBm of the form
 	\begin{align}\label{rough_sde}
 		\dd X_t^H = \sigma(X_t^H) \dd B_t^H, \quad t \in [0,T], \qquad X_0^H = x_0 \in \R^d,
 	\end{align}
have been successfully studied using rough path theory if $H>1/4$ in the work \cite{CoutinQian}. Rough path theory was initiated in the seminal works \cite{TL1,TL2} by T. Lyons in the 1990s. See also the monographs  \cite{FrizVictoir,FrizHairer} for recent exhibitions of this topic. 

In a nutshell, the rough path approach for equation \eqref{rough_sde} for $H>1/3$ can be described as follows:
let  $B[n]^H$ be the piecewise linear $n$-th dyadic approximation of $B^H$, that is
$$ B[n]_t^H= B^H_{t_k} +  \frac{ t- t_k }{t_{k+1}-t_k}   \left(  B^{H}_{t_{k+1}} -     B^{H,}_{t_{k}} \right), \qquad t \in [t_k,t_{k+1}],$$
where $$ t_{k}=  	 \frac{k}{2^n}T, \qquad k=0, \ldots, 2^n, \,\, n \in \mathbb{N}.$$ Then,  for fixed $\omega \in \Omega$, the ordinary differential equation
$$ \dd X[n]^H_t(\omega )= \sigma(X[n]^H_t(\omega )) \dd B[n]_t^H(\omega ) , \quad t \in [0,T], \qquad X^H_0(\omega )=x_0 \in \mathbb{R}^d, $$
has a unique solution, if $\sigma \colon \R^d \rightarrow \R^{d \times m}$ is globally Lipschitz continuous. 
However,   for almost all $\omega \in \Omega$ also the limits
	\begin{align*}%\label{eq:limX}
	X_t^H(\omega) = \lim_{n \to \infty}X[n]^H_t(\omega ), \qquad t \in [0,T], 
\end{align*} 
 exist, if $\sigma \in C_{b}^{3}(\R^d; \R^{d \times m})$. This is a consequence of the universal limit theorem of Lyons, see, e.g., \cite{LyonsNotes},  and the almost sure convergence of 
$$
 		{\bf B}[n]_{s,t}^H =  ({\bf B}^{{\bf 1}}[n]_{s,t}^H ,{\bf B^2}[n]^H_{s,t}), \qquad 0 \leq s < t \leq T,
$$
 	in the $\rho$-H\"older semi-norm for $ 1/3 < \rho <H$,  where
 	$$ {\bf B}^{ \bf 1}[n]_{s,t}^H = B[n]^{H}_t - B[n]^{H}_s $$ 
 	and 
 	$$ {\bf B}^{ \bf 2}[n]^H_{s,t}(i,j) =  \int_s^t \left( B[n]^{H,i}_{\tau} -  B[n]^{H,i}_{s} \right) \dd  B[n]^{H,j}_{\tau}, \qquad i,j=1, \ldots, m. $$   The limit object 	${\bf B}^H=({\bf B}^{{\bf 1},H}, {\bf B}^{{\bf 2},H})$ is called a rough path over $B^H$.
	The $\rho$-H\"older semi-norm mentioned above is given by
 	$$ d_{\rho}( {\bf w}, {\bf v}) =   \sup_{0 \leq s < t \leq T}  \frac{\big \| {\bf w}^{ \bf 1}_{s,t} - {\bf v}^{ \bf 1}_{s,t} \big\|}{|t-s|^{\rho}} +  \sup_{0 \leq s < t \leq T}  \frac{\big \| {\bf w}^{ \bf 2}_{s,t} - {\bf v}^{ \bf2}_{s,t} \big\|}{|t-s|^{2\rho}}
 	$$ for functions ${\bf w}^{\bf 1},{\bf v}^{\bf 1} \colon [0,T] \rightarrow \R^m$ and ${\bf w}^{\bf 2},{\bf v}^{ \bf 2}\colon \{ 0 \leq s \leq t \leq T\} \rightarrow \R^{m \times m}$. Instead of the H\"older semi-norms  one can also use appropriate $p$-variation distances as, e.g., in \cite{CoutinQian,LyonsNotes}.
  
  A crucial difference to the previous cases is that the solution map $\Gamma$ with $X^H=\Gamma(B^H)$ fails to have nice smoothness properties. However, the so-called It\^o--Lyons map, i.e., 	$$ X^{H}=\Gamma^{\textrm{IL}}( \mathbf{B^}^H),$$
  which maps the driving rough path onto the solution, is locally Lipschitz continuous with respect to a suitable rough 
  path topology.  Moreover, all the existence and uniqueness results for SDEs driven by fBm from the previous subsections can be embedded in the rough path framework and yield the same solution, if the SDE coefficients are sufficiently regular. For example, the multi-dimensional case with $H > 1/2$ can be analysed in the rough path framework using the well-defined pathwise Riemann--Stieltjes integrals
 		$$ {\bf B}^{{ \bf 2},H}_{s,t}(i,j) =  \int_s^t \left( B^{H,i}_{\tau} -  B^{H,i}_{s} \right) \dd  B^{H,j}_{\tau}, \qquad i,j=1, \ldots, m. $$

 Now, if the derivative 	 $Y^H = \partial_H X^H$ exists, it should  fulfil equations analoguous  to equations \eqref{deriv_multi-sde_1} and \eqref{deriv_multi-sde_2}. For example, for $d=m=1$ we would expect the derivative to satisfy the SDE
\begin{align*}
 		\dd Y_t^H = \sigma'(X_t^H) Y_t^H \dd B_t^H + \sigma(X_t^H) \dd \big(\partial B_t^H \big), \quad t\in [0,T], \qquad Y_0^H=0.
 \end{align*}
 Thus, to interpret and analyse this equation in a rough path framework, one needs to build a rough path over
 $$ Z^H =  \left( \begin{array} {c} B^H \\ \partial_H B^H \end{array} \right). $$	
However, this cannot be achieved using the dyadic linear approximations of $Z^H$. Theorem 4.3.10 in \cite{Koch} shows that the  dyadic linear approximations of $Z^H$ do not converge.

One can illustrate this result by considering the case  $H =1/2$, $d=1$, in which we would need a  rough path over $(W, \partial_H W)=(B^{1/2}, \partial_H B^{1/2})$. For the diagonal terms we have
 	\begin{align*}
 			\int_s^t \big( W_u[n]-  W_s[n]\big) \dd W_u[n] & = \frac{1}{2} (W_t[n]-W_s[n])^2 \longrightarrow \frac{1}{2} (W_t-W_s)^2 , \\
 				\int_s^t \big( \partial_H W_u[n]- \partial_H  W_s[n]\big) \dd (\partial_H  W_u[n]) & = \frac{1}{2} (\partial_H W_t[n]-\partial_H W_s[n])^2 \longrightarrow\frac{1}{2} (\partial_HW_t[n]-\partial_HW_s[n])^2
 		\end{align*}
 	almost surely for $n \rightarrow \infty$. But the off-diagonal terms as, e.g.,
 		\begin{align*}
 			\int_0^T \big(\partial_H W_u[n]- \partial_H W_0[n]\big) \dd W_u[n] = \frac{1}{2} \sum_{k = 1}^{2^n} \big(\partial_H W_{t_k^n} + \partial_H W_{t_{k-1}^n}\big)\big(W_{t_k^n} -W_{t_{k-1}^n}\big)
 		\end{align*}
 		do not  converge (even) in $L^2(\Omega)$. The correlation of $(W, \partial_H W)$ is  simply too strong.

Thus, the differentiability of equation \eqref{rough_sde} with respect to $H$ remains an open problem. However, we strongly suppose that an analogue to Proposition \ref{lip_exp_d=1} is valid in the general multi-dimensional case  for $H \in (1/3,1)$.
Establishing such a result would require to join the cases $H > 1/2$ and $H \leq 1/2$ and also a very careful tracking of the $H$-dependence of the constants in the Lipschitz estimates for the It\^o--Lyons map $\Gamma^{\textrm{IL}}$, see, e.g., Theorem  4 as well as Theorem 11 and its proof in \cite{BayerSinum}.
 	
 A preliminary result in this direction has been given in Example 44 in \cite{FrizVictoir10} and Theorem 24 in \cite{DeVecchiEtAl}: 
 \begin{theorem}
Let $(H_n)_{n \in \mathbb{N}} \subset (1/3,1/2]^{\mathbb{N}}$ be a sequence such that $H_0=\lim_{n \rightarrow \infty} H_n \in (1/3,1/2]$. Moreover, let 
 $\sigma \in C_b^{3}(\R^d; \R^{d \times m})$. Then, we have that
 $$ X^{H_n} \stackrel{\mathcal{L}}{\longrightarrow} X^{H_0}, \qquad n \rightarrow \infty,$$
in $C^{1/3}([0,T];\R^d)$, that is, we have 
 $$ \mathbf{E} \big[ \Phi(X^{H_n}) \big] \longrightarrow \mathbf{E} \big[\Phi(X^{H_0}) \big], \qquad n \rightarrow \infty, $$
 for all $\Phi \colon C^{1/3}([0,T]; \R^d) \rightarrow  \mathbb{R}$ which are bounded and continuous.
 \end{theorem}

 	\bigskip
 	\bigskip
 	\medskip

 	\section{Further results}
 	
 In this section, we shortly discuss further results from the literature, starting with Wiener integrals with respect to fBm.
 For this, we need to define the space  $\mathcal{E}$ of simple functions  from $[0,T]$ to $\R$ which contains all functions
 $$
 f=\sum_{i=1}^{N} f_{i} \mathbf{1}_{\left[u_{i}, v_{i}\right)}, 
 $$
 with arbitrary $N \in \mathbb{N}$,  $f_{i} \in \R$ and $[u_{i}, v_{i}) \subset [0,T]$, $i=1, \ldots, N$. For these integrands the Wiener integral
with respect to fBm is defined as
 $$
 \mathcal{I}^{H}(f):=\sum_{i=1}^{N} f_{i}(B_{v_{i}}^{H}-B_{u_{i}}^{H}).
 $$
 Exploiting that 
 $$
 \langle f, g \rangle_H= {\mathbf{E}} [I_{1}^{H}(f) I_{1}^{H}(g)], \qquad f,g \in \mathcal{E},
 $$
 defines a scalar product, the domain of $\mathcal{I}^H$ can be extended to integrands from a suitable Hilbert space $\mathcal{L}^{H}$. While for $H=1/2$ one obtains the standard $L^2([0,T];\R)$, for $H \neq 1/2$ the spaces are more involved. 
 
 For  $H<  1/2$ an elegant representation has been obtained in Theorem 2.5 in \cite{bardinajolisWiener}.  In fact, one has
 $$
 \mathcal{L}^{H}=\left\{f \in L^{2}([0, T]; \R):\|f\|_{H}< \infty\right\},
 $$
 where
 \begin{equation*}
 	\|f\|_{H}^2 =\frac{ H(1-2 H)}{2} \int_{0}^{T} \int_{0}^{T} \frac{(f(x)-f(y))^{2}}{|x-y|^{2-2 H}} \dd x \dd y+H \int_{0}^{T} f(x)^{2}\left(\frac{1}{x^{1-2 H}}+\frac{1}{(T-x)^{1-2 H}}\right) \dd x . 
 \end{equation*}
 For $H \geq 1/2$ the space $\mathcal{L}_H$ contains distributions, see, e.g., \cite{jolisDistr}, but we have $L^2([0,T];\R)  \subset \mathcal{L}_H$.

 In both cases the process
 $$  \mathcal{I}_t^{H}(f):= \mathcal{I}_t^{H}(f\mathbf{1}_{\left[0, t\right)}), \qquad t \in [0,T],$$
 has a version with continuous sample paths, whose law depends   continuously on the Hurst parameter; see Theorem 3.5 in \cite{jolis2010continuity} for the case $H_0 \leq 1/2$ and Theorem 3.1 in \cite{jolis2007continuity} for the case $H_0 \geq 1/2$.
 
 \begin{theorem}
 (a)  Let $H_{0} \in (0, 1/2 ]$, $\lambda \in (0,H_0)$ and $H \in (\lambda,1/2]$. Moreover, let $f \in \mathcal{L}^{\lambda}$.  Then, we have  $$ \mathcal{I}^{H}(f) \stackrel{\mathcal{L}}{\longrightarrow} \mathcal{I}^{H_0}(f), \qquad H \rightarrow H_0,$$
 in $C([0,T];\R)$. \\
 (b)	 Let  $H_{0} \in [1/2, 1)$ and $H \in (1/2,1)$. Moreover, let $f \in L^2([0,T];\R)$.Then, we have  $$ \mathcal{I}^{H}(f) \stackrel{\mathcal{L}}{\longrightarrow} \mathcal{I}^{H_0}(f), \qquad H \rightarrow H_0,$$
 in $C([0,T];\R)$.
 \end{theorem}

 For $H \geq 1/2$ this result has been extended to multiple Wiener integrals, see Theorem 4.2 in \cite{jolis2007continuity}. Moreover, the multivariate version of the above theorem is also valid.
 	
 	\smallskip
 	
 The symmetric Russo--Vallois integral has been introduced in 1993, see \cite{RV}, as a generalization of the Stratonovich integral. For a stochastic process $(U^H_t)_{t\in [0,T]}$ whose sample paths are integrable, the symmetric Russo--Vallois integral with respect to fBm is defined as the limit in probability
 $$  \mathcal{I}_T(U^H,B^H):= \lim_{\varepsilon \rightarrow 0}\frac{1}{2 \varepsilon}\int_0^T U_s^H \Big(B^{H}_{s +\varepsilon}-B^{H}_{s-\varepsilon} \Big)  \dd s, $$
 provided that this limit exists. Here  we use the convention that $B_t^H=0$ if $t \notin [0,T]$. The existence of the limit can be verified by Malliavin calculus techniques using the relation between the Russo--Vallois integral and the Skorokhod integral, see, e.g., Proposition 2.2 in \cite{jolis2010symmetric}.  Moreover, the work \cite{OhashiRusso} establishes a relation between the controlled rough path integral and the Russo--Vallois integral for general Gaussian processes as integrators.

 The main result, i.e., Theorem 3.13 of \cite{jolis2010symmetric} is as follows:
 Assume that $H,H_0 \in [1/2,1)$ and
 $$ (U^H,B^H) \stackrel{\mathcal{L}}{\rightarrow} (U^{H_0},B^{H_0}) $$
 in ${C}([0,T];\R) \otimes {C}([0,T];\R)$ for $H \rightarrow H_0$. Then, we also have
  $$ \mathcal{I}_T(U^H,B^H) \stackrel{\mathcal{L}}{\rightarrow} \mathcal{I}_T(U^{H_0},B^{H_0}) $$
 in ${C}([0,T];\R)$ for $H \rightarrow H_0$ under  Malliavin regularity conditions on the process $U^H=(U_t^H)_{t \in [0,T]}$.

\smallskip

The proofs of all the above results for stochastic integrals with respect to fBm rely on arguments involving the convergence of the finite dimensional distributions for $H \rightarrow H_0$ as well as tightness arguments.  Using the same approach,  the continuous dependence of the law of the local time of fBm on the Hurst parameter is established in \cite{jolis2007local}.

\smallskip

Finally, in \cite{GiordanoJolisQS,anderson} stochastic partial differential equations with fractional noise
are considered. For example, in \cite{GiordanoJolisQS} the stochastic heat equation
\begin{align*}
	\frac{\partial}{\partial t} u^{H}(t, x)& =\frac{\partial^{2} }{\partial x \partial x}u^{H}(t, x)+b(u^{H}(t, x))+\frac{\partial^{2} }{ \partial t \partial x}W^{H}(t, x), \\ 
	u^{H}(0, x)&=u_{0}(x),  
\end{align*} on $[0,T] \times \R$
is studied. Here, the initial condition  $u_0\colon \R \rightarrow \R$ and  the drift coefficient $b\colon \mathbb{R} \rightarrow \mathbb{R}$ satisfy suitable assumptions, and $W^H$ is a Gaussian field with covariance
	$$
		\mathbf{E}\big[W^{H}(t, x) W^{H}(s, y)\big]  =\frac{1}{2} \min \{s,t\} \left(|x|^{2 H}+|y|^{2 H}-|x-y|^{2 H}\right), \quad s,t \in [0,T], \,\,x,y \in \R.
	$$
Using the continuity of the solution map of the stochastic heat equation with respect to the driving noise, the authors obtain the following result (Theorem 4.1 in  \cite{GiordanoJolisQS}): 
\begin{theorem}
Let $\alpha \in (0,1)$ and let $(H_n)_{n \in \mathbb{N}} \subset (0,\alpha]^{\mathbb{N}}$ be a sequence such that $H_0=\lim_{n \rightarrow \infty} H_n \in (0,\alpha)$. Moreover, let  $u_0 \in C^{\alpha}(\R;\R)$ and let $b \colon \mathbb{R} \rightarrow \mathbb{R} $ be globally Lipschitz continuous. Then, the above stochastic heat equation has a unique mild solution for all $H_n$, $n \in \mathbb{N}_0$. Moreover, we have $$u^{H_n} \stackrel{\mathcal{L}}{\longrightarrow} u^{H_0}, \qquad n \rightarrow \infty, $$ where the convergence holds in distribution in the space $C([0,T ] \times \R; \R)$ endowed with the metric of uniform convergence on compact subsets.
\end{theorem}

 Similar results hold for the stochastic wave equation with additive fractional noise for $H \in (0,1)$, see also Theorem 4.1 in  \cite{GiordanoJolisQS},  as well as for the hyperbolic and parabolic Anderson model with multiplicative fractional noise for $H \in (1/4,1)$, see Theorem 1.1 in \cite{anderson}.

\bigskip
\bigskip

{\bf Acknowledgements.} \,\,  {\it Andreas Neuenkirch wants to thank the organizers of 
	the CIRM-research school `Stochastic and Deterministic Analysis for Irregular Models' for hosting this instructive event.}

\end{document}